\documentclass[11pt]{amsart}
\usepackage[latin1]{inputenc}
\usepackage{amsmath,amsthm, amscd, amssymb, amsfonts}
\usepackage[dvips]{graphicx}
\numberwithin{equation}{section}

\theoremstyle{plain}

\newtheorem{theorem}{Theorem}[section]
\newtheorem{corollary}[theorem]{Corollary}
\newtheorem{proposition}[theorem]{Proposition}
\newtheorem{lemma}[theorem]{Lemma}
\newtheorem{question}[theorem]{Question}

\theoremstyle{definition}

\theoremstyle{remark}
\newtheorem{remark}[theorem]{Remark}

\newcommand{\C}{{\mathcal C}}
\newcommand{\G}{{\mathcal G}}
\newcommand\id{\operatorname{id}}
\newcommand\Aut{\operatorname{Aut}}
\newcommand\Frat{\operatorname{Frat}}
\newcommand\ad{\operatorname{ad}}
\newcommand\Tr{\operatorname{Tr}}

\newcommand\co{\operatorname{co}}

\newcommand\GL{\operatorname{GL}}
\newcommand\cop{\operatorname{cop}}

\newcommand\Fit{\operatorname{Fit}}
\newcommand\op{\operatorname{op}}
\newcommand\Hom{\operatorname{Hom}}
\newcommand\Rep{\operatorname{Rep}}

\newcommand\Ind{\operatorname{Ind}}
\newcommand\Opext{\operatorname{Opext}}

\newcommand\FPdim{\operatorname{FPdim}}

\newcommand\vect{\operatorname{Vec}}

\newcommand\Irr{\operatorname{Irr}}
\newcommand\cd{\operatorname{c.d.}}

\begin{document}
\title[Irreducible degrees of fusion categories]{On fusion categories with few irreducible degrees}
\author{Sonia Natale}
\author{Julia Yael Plavnik}
\address{Facultad de Matem\'atica, Astronom\'\i a y F\'\i sica,
Universidad Nacional de C\'ordoba, CIEM -- CONICET, (5000) Ciudad
Universitaria, C\'ordoba, Argentina}
\email{natale@famaf.unc.edu.ar, plavnik@famaf.unc.edu.ar
\newline \indent \emph{URL:}\/ http://www.famaf.unc.edu.ar/$\sim$natale}
\thanks{This work was partially supported by CONICET, ANPCyT  and Secyt (UNC)} 
\date{November 4, 2011}

\begin{abstract}
We prove some results on the structure of certain classes of
integral fusion categories and semisimple Hopf algebras under
restrictions on the set of its irreducible degrees.
\end{abstract}

\maketitle

\section{Introduction}

Let $k$ be an algebraically closed field of characteristic zero.
Let $\mathcal C$ be a fusion category over $k$.  That is, $\C$ is
a $k$-linear semisimple rigid tensor category  with finitely many
isomorphism classes of simple objects, finite-dimensional spaces
of morphisms, and such that the unit object $\textbf{1}$ of $\C$
is simple.

For example, if $G$ is a finite group, then the categories $\Rep G$ of its finite dimensional representations, and the category $\C(G, \omega)$ of $G$-graded vector spaces with associativity determined by the $3$-cocycle $\omega$, are fusion categories over $k$. More generally, if $H$ is a finite dimensional semisimple quasi-Hopf algebra over $k$, then the category $\Rep H$ of its finite dimensional representations is a fusion category.

\medbreak Let $\Irr (\C)$ denote the set of isomorphism classes of
simple objects in the fusion category $\C$.   In analogy with the
case of finite groups \cite{isaacs}, we shall use the notation
$\cd (\C)$ to indicate the set
\begin{equation*}\cd(\mathcal C) =  \{ \FPdim x|\, x \in \Irr (\C) \}.
\end{equation*} Here, $\FPdim x$ denotes the \emph{Frobenius-Perron dimension} of $x \in \Irr(\C)$. Notice that, when $\C$ is the representation category of a quasi-Hopf algebra, Frobenius-Perron dimensions coincide with the dimensions of the underlying vector spaces. In this case, we shall use the notation $\cd(\mathcal C) = \cd (H)$.

The positive real numbers $\FPdim x$, $x \in \Irr (\C)$, will be
called the \emph{irreducible degrees} of $\C$.

The fusion categories that we shall consider in this paper are all
\emph{integral}, that is, the Frobenius-Perron dimensions of
objects of $\C$ are (natural) integers. By \cite[Theorem
8.33]{ENO},  $\C$ is isomorphic to the category of representations
of some finite dimensional semisimple quasi-Hopf algebra.

\medbreak For a finite group $G$, the knowledge of the set $\cd(G)
= \cd(kG)$ gives in some cases substantial information about the
structure of $G$.   It is known, for instance, that if $|\cd (G)|
\leq 3$, then $G$ is solvable.

On the other hand, if $|\cd (G)| = 2$, say $\cd(G) = \{ 1, m \}$, $m \geq 1$, then either $G$ has an abelian normal subgroup of index $m$ or else $G$ is nilpotent of class $\leq 3$.
Furthermore, if $G$ is nonabelian, then $\cd(G) = \{ 1, p \}$ for some prime number $p$, if and only if  $G$ contains an abelian normal subgroup of index $p$ or the center $Z(G)$ has index $p^3$. See   \cite[Theorems  (12.11), (12.14) and (12.15)]{isaacs}.

\medbreak In the context of semisimple Hopf algebras, some results in the same spirit are known. A basic one is that of \cite{s-zhu}, which asserts that if $|\cd (H)| \leq 3$, then $G(H^*)$ is not trivial, in other words, $H$ has nontrivial characters of degree $1$.
A similar result appears in \cite[Theorem 2.2.3]{pqq}.

Further results, leading to classification theorems in some specific cases, appear in the work of Izumi and Kosaki \cite{IK} for Kac algebras, that is, Hopf $C^*$-algebras.

\medbreak In this paper we consider the general problem of
understanding the structure of a fusion category $\C$ after the
knowledge of $\cd (\C)$. For instance, it is well-known that
$\cd(\C) = \{ 1\}$ if and only if $\C$ is pointed, if and only if
$\C \simeq \C(G, \omega)$, for some $3$-cocycle $\omega$ on the
group $G = G(\C)$ of isomorphism classes of invertible objects of $\C$.

More specifically, we address the following question:

\begin{question}\label{question} Suppose $\cd(\C) = \{1, p\}$, with $p$ a prime number. What can be said about the structure of $\C$?  \end{question}

We treat mostly structural questions regarding nilpotency and solvability, in the sense introduced in \cite{gel-nik} and \cite{ENO2}. (A related question for semisimple Hopf algebras, that we shall not discuss in the present paper, was posed in \cite[Question 7.2]{Hopfss-IFUM}.)

\medbreak The notions of nilpotency and solvability of a fusion
category are related to the corresponding notions for
finite groups as follows: if $G$ is a finite group, then the
category $\Rep G$ is nilpotent or solvable if and only if $G$ is
nilpotent or solvable, respectively. On the dual side,  a pointed
fusion category $\C(G, \omega)$ is always nilpotent, while it is
solvable if and only if the group $G$ is solvable.

An important class of fusion categories, called \emph{weakly
group-theoretical} fusion categories, was introduced  and studied
in \cite{ENO2}. This generalized in turn the notion of a
group-theoretical fusion category of \cite{ENO}. By definition, $\C$ is
group-theoretical if it is Morita equivalent to a pointed fusion
category, and it is weakly group-theoretical if it is Morita equivalent to a nilpotent fusion
category. Every nilpotent or solvable fusion category
is weakly group-theoretical.

\medbreak With regard to Question \ref{question}, consider for instance the case where $\C = \Rep H$, for a semisimple Hopf algebra $H$.
A result in this direction is known in the case $p = 2$. It is shown in \cite[Corollary 6.6]{BN} that if $H$ is a semisimple Hopf algebra such that $\cd(H) \subseteq \{ 1, 2 \}$, then $H$ is upper semisolvable. Moreover, $H$ is necessarily cocommutative if $G(H^*)$ is of order $2$.
The proof of these results relies on a refinement of a theorem of Nichols and Richmond (\cite[Theorem 11]{NR}) given in \cite[Theorem 1.1]{BN}.

In the context of Kac algebras,  it is shown in \cite[Theorem IX.8 (iii)]{IK} that if $\cd (H^*) = \{1, p\}$ and in addition $|G(H)| = p$, then $H$ is a central abelian extension associated to an action of the cyclic group of order $p$ on a nilpotent group. In the recent terminology introduced in \cite{gel-nik}, this result implies that such a Kac algebra is \emph{nilpotent}. See Remark \ref{i-k}.

\medbreak The main results of this paper are summarized in the following theorem.

\begin{theorem} Let $\C$ be a fusion category over $k$. Then we have:

(i) Suppose $\C$ is weakly group-theoretical and has odd
dimension. Then $\C$ is solvable. (Proposition \ref{odd-wgt}.)

\medbreak Let $p$ be a prime number.

\medbreak (ii) Suppose $\C$ is odd-dimensional and $\cd(\C)
\subseteq \{p^m:\, m \geq 0\}$. Then $\C$ is solvable. (Theorem
\ref{braided-odd}.)

\medbreak (iii) Suppose $\cd(\C) \subseteq \{1, p\}$. Then $\C$ is
solvable in any of the following cases:
\begin{itemize}\item $\C = \C(G, \omega, \mathbb Z_p, \alpha)$ (a group-theoretical fusion category \cite{ENO}) and $G(\C)$ is of order $p$. (Corollary \ref{solv-gt}.)

\item $\C$ is a near-group category \cite{Siehler}. (Theorem \ref{solv-neargp}.)

\item $\C = \Rep H$, where $H$ is a semisimple quasitriangular Hopf algebra and $p = 2$. (Theorem \ref{solv-2}.)
\end{itemize}

\medbreak (iv) Let $H$ be a semisimple Hopf algebra such that
$\cd(H) \subseteq \{1, p\}$. Then $H^*$ is nilpotent in any of the
following cases:
\begin{itemize}\item $|G(H^*)| = p$ and $p$ divides $|G(H)|$. (Proposition \ref{dual-nilp}.)

\item $|G(H^*)| = p$ and $H$ is quasitriangular. (Proposition \ref{qt-1p}.)

\item $H$ is of type $(1, p; p, 1)$ as an algebra. (Proposition \ref{type1pp1}.)
\end{itemize}

(v) Let $H$ be a semisimple Hopf algebra such that  $\cd(H)
\subseteq \{ 1, 2\}$. Then we have:
\begin{itemize} \item $H$ is weakly group-theoretical, and furthermore, it is group-theoretical if $H = H_{\ad}$. (Theorem \ref{2-wgt}.)
\item The group $G(H)$ is solvable. (Corollary \ref{2-gdeh}.)
\end{itemize}

(vi) Let $H$ be a semisimple Hopf algebra of type $(1, p; p, 1)$
as an algebra. Then $H$ is isomorphic to a twisting of the group algebra $kN$, where either $p = 2$ and $N = \mathbb S_3$ or $p = 2^{\alpha -1}$, $\alpha > 1$, and $N$ is the affine group of the field $\mathbb F_{2^\alpha}$.
(Theorem \ref{twist-aff}.)
\end{theorem}

The proof of part (i) of the theorem is a consequence of the Feit-Thompson Theorem \cite{feit-thompson} that asserts that every finite group of odd order is solvable.

By \cite[Corollary 4.5]{Hopfss-IFUM}, the semisimple Hopf algebras
$H$ in part (iv) of the theorem are \emph{lower semisolvable} in
the sense of \cite{MW}.

\medbreak The results on semisimple Hopf algebras $H$ with $\cd(H) \subseteq \{ 1, 2\}$ rely on the results of the paper \cite{BN}.
We also make  strong use of several results of the
papers \cite{ENO2}, \cite{gel-nik} and \cite{gel-naidu} on weakly
group-theoretical, solvable and nilpotent fusion categories.

\medbreak The paper is organized as follows.  In Section \ref{two}
we recall the  main notions and results relevant to the problem we
consider. In particular, several properties of group-theoretical
fusion categories and Hopf algebra extensions are  discussed here.
The results on nilpotency are contained in Sections \ref{three}
and \ref{four}. The strategy in these sections consists in
reducing the problem to considering Hopf algebra extensions.
Sections \ref{solvability}, \ref{six} and \ref{seven} are devoted
to the solvability question in different situations.

\section{Preliminaries}\label{two}

\subsection{Fusion categories}\label{fusion}

A \emph{fusion category} over $k$ is a $k$-linear semisimple rigid
tensor category $\C$ with finitely many isomorphism classes of
simple objects, finite-dimensional spaces of morphisms, and such
that the unit object $\textbf{1}$ of $\C$ is simple. We refer the
reader to \cite{BK, ENO} for basic definitions and facts
concerning fusion categories. In particular, if $H$ is a
semisimple (quasi-)Hopf algebra over $k$, then $\Rep H$ is a
fusion category.

A \emph{fusion subcategory} of a fusion category $\C$
is a full tensor subcategory $\C^{'} \subseteq \C$ such
that if $X \in \C$ is isomorphic to a direct summand of an object
of $\C^{'}$, then $X \in \C^{'}$.  A fusion subcategory is necessarily
rigid, so it is indeed a fusion category \cite[Corollary
F.7 (i)]{DGNO1}.

\medbreak A \emph{pointed fusion  category} is a fusion category where all  simple objects
are invertible. A pointed fusion category is equivalent to the
category $\C(G, \omega)$,  of finite-dimensional $G$-graded vector spaces with associativity
constraint determined by a cohomology class $\omega\in
H^3(G,k^×)$, for some finite group $G$.
In other words, $\C(G, \omega)$ is the category of representations
of the quasi-Hopf algebra $k^G$, with associator $\omega\in
(k^G)^{\otimes 3}$.

\medbreak The fusion subcategory \emph{generated} by a collection $\mathcal X$ of objects of $\C$ is the smallest
fusion subcategory containing $\mathcal X$.

If $\C$ is a fusion category, then the set of isomorphism classes of invertible objects of $\C$ forms a group, denoted $G(\C)$. The fusion
subcategory  generated by  the invertible objects of $\C$ is a fusion subcategory, denoted $\C_{pt}$; it is the maximal pointed subcategory of $\C$.

\medbreak Let $\Irr (\C)$ denote the set of isomorphism classes of simple objects in the fusion category $\C$.  The set $\Irr (\C)$ is a basis over $\mathbb Z$ of the Grothendieck  ring $\G(\C)$.

\subsection{Irreducible degrees}\label{irr-fusion} For  $x \in \Irr(\C)$, let $\FPdim x$ be its Frobenius-Perron dimension.
The positive real numbers $\FPdim x$, $x \in \Irr (\C)$ will be
called the \emph{irreducible degrees} of $\C$. These extend to a
ring homomorphism $\FPdim: \G(\C) \to \mathbb R$. When $\C$ is the
representation category of a quasi-Hopf algebra, Frobenius-Perron
dimensions coincide with the dimension of the underlying vector
spaces.

The set of \emph{irreducible degrees} of $\C$ is defined as
\begin{equation*}\cd(\mathcal C) =  \{ \FPdim x|\, x \in \Irr (\C) \}.
\end{equation*}

The category $\C$ is called \emph{integral} if $\cd(\C) \subseteq
\mathbb N$.

\medbreak If $X$ is any object of $\C$, then its class $x$  in
$\G(\C)$ decomposes as $x = \sum_{y \in \Irr (\C)} m(y, x) y$,
where  $m(y, x) = \dim \Hom(Y, X)$ is the multiplicity of $Y$ in
$X$, if $Y$ is an object representing the class $y \in \Irr(\C)$.

For all $x, y, z \in \G(\C)$, we have:
\begin{equation}\label{adj-c}m(x, yz) = m(y^*, z x^*) = m(y, x z^*). \end{equation}

Let $x \in \Irr (\C)$. The stabilizer of $x$ under left
multiplication by elements of $G(\C)$ in the Grothendieck ring
will be denoted by $G[x]$. So that an invertible element $g \in
G(\C)$ belongs to $G[x]$ if and only if $g x = x$.

In view of \eqref{adj-c}, for all $x \in \Irr (\C)$, we have
\begin{equation*}G[x]= \{ g \in G(\C):\, m(g, x x^*) > 0 \} =  \{
g \in G(\C):\, m(g, x x^*) = 1 \}.\end{equation*} In particular,
we have the following relation in $\G(\C)$:
\begin{equation*}x x^* = \sum_{g \in G[x]} g +
\sum_{y \in \Irr(\C),\, \FPdim y > 1} m(y, x x^*) y.\end{equation*}

\begin{remark} Notice that an object $g \in \C$ is invertible if and only if $\FPdim g = 1$.

\medbreak Suppose that $\C$ is an integral fusion category with
$|\cd (\C)| = 2$. That is, $\cd (\C) = \{1, d\}$ for some integer
$d > 1$. Then $d$ divides the order of $G[x]$ for all  $x \in
\Irr(\C)$ with $\FPdim x > 1$. In particular, $d$ divides the
order of $G(\C)$, and thus $G(\C) \neq 1$.

Indeed, if $x \in \Irr(\C)$ with $\FPdim x = d$, we have the
relation:
$$x x^* = \sum_{g\in G[x]} g + \underset{y \in \Irr(\C),\, \FPdim y = d}{\sum}
m(y, x x^*)y.$$ The claim follows by taking Frobenius-Perron
dimensions. \end{remark}

\subsection{Semisimple Hopf algebras} Let $H$ be a semisimple Hopf algebra over $k$. We next recall some
of the terminology and conventions from \cite{ssld} that will be
used throughout this paper.

\medbreak As an algebra, $H$ is isomorphic to a direct sum of full
matrix algebras
\begin{equation}\label{estructura} H \simeq k^{(n)} \oplus \oplus_{i = 1}^r
M_{d_i}(k)^{(n_i)},\end{equation} where $n = |G(H^*)|$. The
Nichols-Zoeller theorem \cite{NZ} implies that $n$ divides both
$\dim H$ and $n_i d_i^2$, for all $i = 1, \dots, r$.

\medbreak If we have an isomorphism as in \eqref{estructura},  we
shall say that $H$ is \emph{of type} $(1, n; d_1, n_1; \dots; d_r,
n_r)$ \emph{as an algebra}. If $H^*$ is of type $(1, n; d_1, n_1;
\dots; d_r,
n_r)$ as an algebra, we shall say that $H$ is  \emph{of type}
$(1, n; d_1, n_1; \dots; d_r,
n_r)$ \emph{as a coalgebra}.

\medbreak Let $V$ be an $H$-module. The \emph{character} of $V$ is
the element $\chi = \chi_V \in H^*$ defined by $\chi(h) =
\Tr_V(h)$, for all $h \in H$. For a character $\chi$, its
\emph{degree} is the integer $\deg \chi = \chi(1) = \dim V$. The
character $\chi_V$ is called irreducible if $V$ is irreducible.

\medbreak The set $\Irr(H)$ of irreducible characters of
$H$ spans a semisimple subalgebra $R(H)$ of $H^*$, called the
character algebra of $H$. It is isomorphic, under the map $V \to
\chi_V$, to the extension of scalars $k \otimes_{\mathbb Z}\G(\Rep
H)$ of the Grothendieck ring of the category $\Rep H$. In
particular, there is an identification $\Irr (H) \simeq \Irr (\Rep
H)$.

Under this identification, all properties listed in Subsection \ref{irr-fusion} hold true for characters.

In this context,  we have $G(\Rep H) = G(H^*)$. The stabilizer of
$\chi$ under left multiplication by elements in $G(H^*)$ will be
denoted by $G[\chi]$. By the Nichols-Zoeller
theorem \cite{NZ}, we have that $|G[\chi]|$ divides $(\deg
\chi)^2$.

\medbreak Following \cite[Chapter 12]{isaacs}, we shall denote
$\cd(H) = \cd(\Rep H)$. So that, $$\cd(H) = \{ \deg \chi|\, \chi
\in \Irr(H) \}.$$ In particular, if $H$ is of type $(1, n; d_1,
n_1; \dots; d_r, n_r)$ as an algebra, then $\cd(H) = \{ 1, d_1,
\dots, d_r\}$.

\medbreak There is a bijective correspondence between Hopf algebra
quotients of $H$ and standard subalgebras of $R(H)$, that is,
subalgebras spanned by irreducible characters of $H$. This
correspondence assigns to the Hopf algebra quotient $H \to
\overline H$ its character algebra $R(\overline H) \subseteq
R(H)$. See \cite{NR}.

\subsection{Group-theoretical categories}\label{group-ttic}

A group-theoretical fusion category is a fusion category Morita
equivalent to a pointed fusion category $\C(G, \omega)$. Such a
fusion category is equivalent to the category $\C(G, \omega, F,
\alpha)$ of $k_{\alpha}F$-bimodules in $\C(G, \omega)$, where $G$
is a finite group, $\omega$ is a $3$-cocycle on $G$, $F \subseteq
G$ is a subgroup and $\alpha \in C^2(F, k^{\times})$ is a
$2$-cochain on $F$ such that $\omega\vert_F = d\alpha$. A
semisimple Hopf algebra $H$ is called group-theoretical if the
category $\Rep H$ is group-theoretical.

\medbreak Let $\C = \C(G, \omega, F, \alpha)$ be a group-theoretical fusion category. Let also $\Gamma$ be a subgroup of $G$, endowed with a $2$-cocycle $\beta\in Z^2(\Gamma,k^×)$, such that
\begin{itemize}
\item The class $\omega|_ \Gamma$ is trivial;
\item $G = F\Gamma$;
\item The class $\alpha|_{F\cap\Gamma}\beta^{-1}|_{F\cap\Gamma}$ is non-degenerate.
\end{itemize}
Then there is an associated semisimple Hopf algebra $H$, such that the category $\Rep H$ is equivalent to $\C$. By \cite{ostrik}, equivalence classes subgroups $\Gamma$ of $G$ satisfying the conditions above, classify fiber functors $\C \mapsto \vect_k$; these correspond to the distinct Hopf algebras $H$.

\medbreak Let $\C = \C(G, \omega, F, \alpha)$ be a
group-theoretical fusion category. The simple objects of $\C$ are
classified by pairs $(s, U_s)$, where $s$ runs over a set of
representatives of the double cosets of $F$ in $G$, that is,
orbits of the action of $F$ in the space $F\backslash G$ of left
cosets of $F$ in $G$,  $F_s = F \cap sFs^{-1}$ is the stabilizer
of $s \in F\backslash G$, and $U_s$ is an irreducible
representation of the twisted group algebra $k_{\sigma_s}F_s$,
that is, an irreducible projective representation of $F_s$ with
respect to certain $2$-cocycle $\sigma_s$ determined by $\omega$.
See \cite[Theorem 5.1]{gel-naidu}.


\medbreak The irreducible representation $W_{(s, U_s)}$ corresponding to such a pair $(s, U_s)$ has dimension  \begin{equation}\dim W_{(s, U_s)} = [F: F_s] \dim U_s. \end{equation}

As a consequence, we have:

\begin{corollary}\label{irr-abext} The irreducible degrees of $\C(G, \omega, F, \alpha)$ divide the order of $F$. \end{corollary}

\begin{remark} A group-theoretical category $\C = \C(G, \omega, F, \alpha)$ is an integral fusion category. An explicit construction of a quasi-Hopf algebra $H$ such that $\Rep H \simeq \C$ was given in \cite{fs-indic}.

As an algebra, $H$ is a crossed product $k^{F\backslash G}\#_{\sigma}kF$, where $F\backslash G$ is the space of left cosets of $F$ in $G$ with the natural action of $F$, and $\sigma$ is a certain $2$-cocycle determined by $\omega$.

Irreducible representations of $H$, that is, simple objects of $\C$, can therefore be described using the results for
group crossed products in \cite{MW}: this is done in \cite[Proposition 5.5]{fs-indic}.
\end{remark}

By \cite[Theorem 5.2]{gel-naidu}, the group $G(\C)$ of invertible objects of $\C$ fits into an exact sequence \begin{equation}\label{inv-gt}1 \to \widehat F \to G(\C) \to K \to
1,\end{equation} where $K = \{ x \in N_G(F): \, [\sigma_x] = 1 \}$.

\subsection{Abelian extensions}\label{Abelian-extensions}

Suppose that $G = F\Gamma$ is an exact factorization of the finite
group $G$, where $\Gamma$ and $F$ are subgroups of $G$.
Equivalently, $F$ and $\Gamma$ form a \emph{matched pair} of
groups with the actions $\vartriangleleft : \Gamma \times F \to
\Gamma$, $\vartriangleright : \Gamma \times F \to F$, defined by
$sx = (x \vartriangleleft s)(x \vartriangleright s)$, $x \in F$,
$s \in \Gamma$. In this case, $G$ is isomorphic to the group $F\bowtie \Gamma$ defined as follows: $F\bowtie \Gamma = F \times \Gamma$,  with multiplication
$(x, s) (t, y) = (x(s \vartriangleright y), (s \vartriangleleft y) t)$, for all $x, y \in F$, $s, t \in \Gamma$.

Let $\sigma\in Z^2(F, (k^\Gamma)^{\times})$ and $\tau \in Z^2(\Gamma,
(k^F)^{\times})$ be normalized 2-cocycles with the respect to the actions afforded,
respectively, by $\vartriangleleft$ and $\vartriangleright$,
subject to appropriate compatibility conditions \cite{ma-ext}.

\medbreak The bicrossed product $H = k^\Gamma \,
{}^{\tau}\#_{\sigma}kF$ associated to this data is a semisimple
Hopf algebra.
There is an \emph{abelian} exact sequence
\begin{equation}\label{abelian-sec} k \to k^{\Gamma} \to H \to kF \to k. \end{equation} Moreover, every Hopf algebra $H$
fitting into such an exact sequence can be described in this way.
This gives a bijective correspondence between the equivalence
classes of Hopf algebra extensions \eqref{abelian-sec} associated to the matched pair $(F,
\Gamma)$ and a certain abelian group $\Opext (k^\Gamma, kF)$.

\begin{remark} \label{H-group-theoretical}
The Hopf algebra $H$ is group theoretical. In fact, we have an
equivalence of fusion categories $\Rep H \simeq \C(G, \omega, F,
1)$ \cite[4.2]{gp-ttic}, where $\omega$ is the $3$-cocycle on $G$
coming from the so called \emph{Kac exact sequence}.
\end{remark}

Irreducible representations of $H$ are classified by pairs $(s, U_s)$, where $s$ runs over a set of representatives of the orbits of the action of $F$ in $\Gamma$,  $F_s = F \cap sFs^{-1}$ is the stabilizer
of $s \in  \Gamma$, and $U_s$ is an irreducible representation of the twisted group algebra $k_{\sigma_s}F_s$, that is, an irreducible projective representation of $F_s$ with cocycle $\sigma_s$, where $\sigma_s(x, y) = \sigma(x, y)(s)$, $x, y \in F$, $s \in \Gamma$. See \cite{KMM}.

Note that, for all $s\in \Gamma$, the restriction of $\sigma_s: F
\times F \to k^{\times}$ to the stabilizer $F_s$ defines indeed a
$2$-cocycle on $F_s$.

\medbreak The irreducible representation corresponding to such a
pair $(s, U_s)$ is in this case of the form \begin{equation}W_{(s,
U_s)} : = \Ind_{k^{\Gamma}\otimes k F_s}^H s\otimes
U_s.\end{equation}


\subsection{Quasitriangular Hopf algebras}\label{quasitriangular}

Let $H$ be a finite dimensional Hopf algebra. Recall that $H$ is called \emph{quasitriangular} if there exists  an invertible element  $R \in H \otimes
H$, called an \emph{$R$-matrix}, such that
\begin{align*}& (\Delta \otimes \id)(R) = R_{13}R_{23},  \quad (\epsilon \otimes \id)(R) = 1, \\
& (\id \otimes \Delta)(R) = R_{13}R_{12},  \quad (\id \otimes \epsilon)(R) = 1, \\
& \Delta^{\cop}(h)  = R \Delta(h) R^{-1}, \quad \forall h \in H.\end{align*}

The existence of an $R$-matrix (also called a \emph{quasitriangular structure} in what follows) amounts to the category $\Rep H$ being a braided tensor category. See \cite{BK}.

For instance, the group algebra $kG$ of a finite group $G$ is a
quasitriangular Hopf algebra with $R = 1\otimes 1$. On the other
hand, the dual Hopf algebra $k^G$ admits a quasitriangular
structure if and only if $G$ is abelian.

If it exists, a quasitriangular structure in a Hopf algebra $H$
need not be unique.

\medbreak Another example of a quasitriangular Hopf algebra is the \emph{Drinfeld
double} $D(H)$ of $H$, where $H$ is any finite dimensional Hopf algebra. We have $D(H) = H^{* \cop}
\otimes H$ as coalgebras, with a canonical $R$-matrix $\mathcal R
= \sum_i h^i \otimes h_i$, where $(h_i)_i$ is a basis of $H$ and
$(h^i)_i$ is the dual basis.

As braided tensor categories, $\Rep D(H) = \mathcal Z(\Rep H)$ is isomorphic to the center of
the tensor category $\Rep  H$.

\medbreak Suppose $(H, R)$ is a quasitriangular Hopf algebra. There are Hopf algebra maps $f_R: {H^*}^{\cop} \to H$ and
$f_{R_{21}}:H^* \to H^{\op}$ defined by
$$f_R(p) = p(R^{(1)}) R^{(2)},
\quad f_{R_{21}}(p) = p(R^{(2)}) R^{(1)},$$ for all $p \in H^*$,
where $R = R^{(1)} \otimes R^{(2)} \in H \otimes H$.

We shall denote $f_R(H^*) = H_+$ and $f_{R_{21}}(H^*) = H_-$, respectively.
So that $H_+, H_-$ are Hopf subalgebras of $H$ and we have
$H_+ \simeq (H_-^*)^{\cop}$.

We shall also denote by $H_R = H_-H_+ = H_+H_-$ the minimal
quasitriangular Hopf subalgebra of $H$. See \cite{R}.

\medbreak By \cite[Theorem 2]{R}, the multiplication of $H$ determines a surjective Hopf algebra map $D(H_-) \to H_R$.

\medbreak A quasitriangular Hopf algebra
$(H, R)$ is called {\it factorizable} if the map $\Phi_R : H^* \to
H$ is an isomorphism, where
\begin{equation}\label{phir} \Phi_R (p) = p(Q^{(1)}) Q^{(2)}, \quad p \in H^*;
\end{equation}
here, $Q = Q^{(1)} \otimes Q^{(2)} = R_{21} R \in H \otimes H$ \cite{re-se}.

If on the other hand $\Phi_R = \epsilon 1$ (or equivalently,
$R_{21}R = 1 \otimes 1$), then $(H, R)$ is called
\emph{triangular}. Finite dimensional triangular Hopf algebras
were completely classified in \cite{eg-triangular}. In particular, if $(H, R)$ is a semisimple quasitriangular Hopf algebra, then $H$ is isomorphic, as a Hopf algebra, to a twisting $(kG)^J$ of some finite group $G$.

It is well-known that the Drinfeld double $(D(H), \mathcal R)$ is
indeed a \emph{factori\-zable} quasitriangular Hopf algebra. We have $D(H)_+ = H$, $D(H)_- =
{H^*}^{\cop}$.

\medbreak We shall use later on in this paper the following result
about factorizable Hopf algebras. A categorical version is
established in \cite{gel-nik}.

\begin{theorem}\label{central-gplk}\cite[Theorem 2.3]{schneider}. Let $(H, R)$ be a factorizable Hopf algebra. Then the map $\Phi_R$ induces an isomorphism of groups $G(H^*) \to G(H)\cap Z(H)$.\end{theorem}

Note that we may identify $G(D(H)) = G(H^*) \times G(H)$. Under this identification, Theorem \ref{central-gplk} gives us a group isomorphism $G(D(H)^*) \to G(D(H)) \cap Z(D(H))$, such that $g \# f \mapsto f \# g$. See also \cite{R}.

In particular, if $f = \epsilon$, then $g \in G(H) \cap Z(H)$, and also if $g = 1$, then $f \in G(H^*)\cap Z(H^*)$.

\medbreak Suppose $(H, R)$ is a finite dimensional quasitriangular
Hopf algebra, and let $D(H)$ be the Drinfeld double of $H$.  Then
there is a surjective Hopf algebra map $f: D(H) \to H$, such that $(f\otimes f)
\mathcal R = R$. The map $f$ is determined by $f(p\otimes h) = f_R(p)h$, for all $p \in H^*$, $h \in H$.

This corresponds to the canonical inclusion of the braided tensor category $\Rep H$ (with
braiding determined by the action of the $R$-matrix) into its
center.

In particular, in the case where $H$ is quasitriangular, the group $G(H^*)$ can be identified with a subgroup of $G(D(H)^*)$.

\section{Nilpotency}\label{three}

Let $G$ be a finite group. A \emph{$G$-grading} of a fusion category $\C$
is a decomposition of $\C$ as a direct sum of full abelian
subcategories $\C = \oplus_{g \in G} \C_g$, such that $\C_g^* =
\C_{g^{-1}}$ and  the tensor product $\otimes: \C \times \C \to
\C$ maps $\C_g \times \C_h$ to $\C_{gh}$.  The neutral component
$\C_e$ is thus a fusion subcategory of $\C$.

The grading is called \emph{faithful} if $\C_g \neq 0$, for all $g
\in G$. In this case, $\C$ is called a \emph{$G$-extension} of
$\C_e$ \cite{ENO2}.

\medbreak The following proposition is a consequence of
\cite[Theorem 3.8]{gel-nik}.

\begin{proposition}\label{grading-hopf} Let $\C = \Rep H$, where
$H$ is a semisimple Hopf algebra.  Then a faithful $G$-grading on
$\C$ corresponds to a \emph{central} exact sequence of Hopf
algebras $k \to k^G \to H \to \overline H \to k$, such that $\Rep
\overline H = \C_e$. \end{proposition}

Let $\C$ be a fusion category and let $\C_{\ad}$   be the adjoint
subcategory of $\C$. That is, $\C_{\ad}$ is the fusion subcategory
of $\C$ generated by $X \otimes X^*$, where $X$ runs through the
simple objects of $\C$.

It is shown in \cite{gel-nik} that there is a canonical faithful
grading on $\C$: $\C = \oplus_{g \in U(\C)}\C_g$, called the
\emph{universal grading}, such that $\C_e = \C_{\ad}$. The group
$U(\C)$ is called the \emph{universal grading group} of $\C$.

\medbreak In the case where $\C = \Rep H$, for a  semisimple Hopf
algebra $H$, $K = k^{U(\C)}$ is the maximal central Hopf
subalgebra of $H$ and $\C_{\ad} = \Rep H/HK^+$ \cite[Theorem 3.8]{gel-nik}.



\medbreak Recall from \cite{gel-nik, ENO2} that a fusion category
$\C$ is called (cyclically) \emph{nilpotent} if  there is a sequence of fusion
categories $$\C_0 = \vect, \C_1, \dots , \C_n = \C$$ and a
sequence $G_1, \dots , G_n$ of finite (cyclic) groups such that $\C_i$ is
faithfully graded by $G_i$ with trivial component $\C_{i-1}$.

\medbreak The semisimple Hopf algebra $H$ is called nilpotent if
the fusion category $\Rep H$ is nilpotent \cite[Definition
4.4]{gel-nik}.

For instance, if $G$ is a finite group, then the dual group algebra $k^G$ is always nilpotent. However, the group algebra $kG$ is nilpotent if and only if the group $G$ is nilpotent \cite[Remark 4.7. (1)]{gel-nik}.

\subsection{Nilpotency of an abelian extension}

It is shown in \cite[Corollary 4.3]{gel-naidu} that a
group-theoretical fusion category $\C(G, \omega, F, \alpha)$  is
nilpotent if and only if the normal closure of $F$ in $G$ is
nilpotent. On the other hand, this happens if and only if $F$ is nilpotent and  subnormal in
$G$, if and only if $F \subseteq \Fit(G)$, where $\Fit(G)$ is the
Fitting subgroup of $G$, that is, the unique
largest normal nilpotent subgroup of $G$ \cite[Subsection 2.3]{gel-naidu}.

Combined with Remark \ref{H-group-theoretical}, this implies:

\begin{proposition}\label{abelian-nilpotent} Let $k \to k^{\Gamma} \to H \to kF \to k$ be an abelian exact sequence
and let $G = F\bowtie \Gamma$ be the associated factorizable
group. Then $H$ is nilpotent if and only if $F\subseteq \Fit(G)$.
\end{proposition}

\medbreak An abelian exact sequence \eqref{abelian-sec} is called \emph{central} if the image of $k^{\Gamma}$ is a central Hopf subalgebra of $H$. It is called cocentral, if the dual exact sequence is central.

The following facts are well-known:

\begin{lemma}\label{lema-central} Consider an abelian exact sequence \eqref{abelian-sec}. Then we have:

\begin{itemize}\item[(i)]  The sequence is central if and only if the action $\vartriangleleft: \Gamma \times F \to \Gamma$ is trivial. In this case, the group $G = F\bowtie\Gamma$ is a semidirect product $G \simeq F \rtimes \Gamma$ with respect to the action $\vartriangleright: \Gamma \times F \to F$.

\item[(ii)] The sequence is cocentral if and only if the action $\vartriangleright: \Gamma \times F \to F$ is trivial. In this case, the group $G = F\bowtie\Gamma$ is a semidirect product $G \simeq F \ltimes \Gamma$ with respect to the action $\vartriangleleft: \Gamma \times F \to \Gamma$.\qed \end{itemize}  \end{lemma}


\begin{remark}\label{central-nilpotent} Assume the exact sequence \eqref{abelian-sec} is central. Then $F$ is a normal subgroup of $G$. It follows from Proposition \ref{abelian-nilpotent} that $H$ is nilpotent if and only if  $F$ is nilpotent. \end{remark}

\section{On the nilpotency of a class of semisimple Hopf algebras}\label{four}

Let $p$ be a prime number. We shall consider in this subsection a nontrivial semisimple Hopf
algebra $H$ fitting into an abelian exact sequence
\begin{equation}\label{abext-1p}k \to k^{\mathbb Z_p} \to H \to kF
\to k.\end{equation}
The main result of this subsection is Proposition \ref{Fnilp} below.

\medbreak We first have the following lemma.

Suppose that $\C$ is any group-theoretical fusion category of the
form $\C = \C(G, \omega, \mathbb Z_p, \alpha)$ (Note that we may
assume that $\alpha = 1$.) In particular, $p$ divides the order of
$G(\C)$. We also have $\cd(\C) \subseteq \{ 1, p\}$, by Corollary
\ref{irr-abext}.

\begin{lemma}\label{Fnilp-gt} Let $\C = \C(G, \omega, \mathbb Z_p, \alpha)$. Assume that   $|G(\C)| = p$. Then  $G$ is a Frobenius group with Frobenius complement $\mathbb Z_p$.
 \end{lemma}

\begin{proof}
The description of the irreducible representations of $\C$ in Subsection \ref{group-ttic}, combined with the assumption that $|G(\C)| = p$,
implies that $g\,\mathbb{Z}_p \,g^{-1}\cap
\mathbb{Z}_p = \{e\}$, for all $g \in G\backslash \mathbb Z_p$. (In particular, the action of $\mathbb Z_p$ on $\mathbb Z_p\backslash G$ has no fixed
points $s \neq e$.)

This condition says that $G$ is a Frobenius group
with Frobenius complement $\mathbb Z_p$, as claimed. \end{proof}

\begin{remark}\label{rmk-gt} Let $G$ be a Frobenius group with Frobenius complement $\mathbb Z_p$,  as in Lemma \ref{Fnilp-gt}. By Frobenius' Theorem we
have that  the Frobenius kernel $N$ is a normal subgroup of $G$,
such that $G$ is a semidirect product $G = N\rtimes \mathbb{Z}_p$.
Moreover, $N$ is a nilpotent group, by a theorem of Thompson. See
\cite[Theorem 10.5.6]{robinson}, \cite[Theorem (7.2)]{isaacs}. In fact, the Frobenius kernel $N$ is equal to $\Fit(G)$, the Fitting subgroup of $G$ \cite[Exercise 10.5.8]{robinson}. \end{remark}

As a consequence we get the following:

\begin{proposition}\label{Fnilp} Consider the abelian exact sequence \eqref{abext-1p} and assume that   $|G(H)| = p$.
Then we have:
\begin{enumerate}\item[(i)] The sequence is central, that is, $G(H) \subseteq Z(H)$.
\item[(ii)] $G = F\bowtie \mathbb Z_p$ is a Frobenius group with kernel $F$. In particular, $F$ is nilpotent.
\end{enumerate}
 \end{proposition}

\begin{proof}
We follow the lines of the proof of \cite[Proposition X.7 (i)]{IK}.
Consider the matched pair $(F, \mathbb{Z}_p)$ associated  to \eqref{abext-1p}, as in Subsection \eqref{Abelian-extensions}. Let $G = F \bowtie \mathbb Z_p$ be the corresponding factorizable group.

We have an equivalence of fusion categories $\Rep H^* \simeq \C(G, \omega, \mathbb Z_p, 1)$, see Remark \ref{H-group-theoretical}. Then $\Rep H^*$ is group-theoretical and by assumption, $G(\Rep H^*)$ is of order $p$. By Lemma \ref{Fnilp-gt}, $G$ is a Frobenius group
with Frobenius complement $\mathbb Z_p$. Therefore $G$ is a semidirect product $G = N\rtimes \mathbb{Z}_p$, where $N = \Fit(G)$ is a nilpotent subgroup (see Remark \ref{rmk-gt}).

\medbreak Since $|G(H)| = p$, then the action of $\mathbb Z_p$ on $F$ has no fixed points. It follows, after decomposing $F$ as a disjoint union of
$\mathbb Z_p$-orbits, that $|F| = 1 \, (\textrm{mod }p)$. In
particular, $|F|$ is not divisible by $p$. Then $F$ must map
trivially under the canonical projection $G \to G/N$, that is, $F
\subseteq N$. Hence $F = N$,  because they have the same order.
This shows (ii). Since $F$ is normal in $G$, we get (i) in view of
Lemma \ref{lema-central}.
\end{proof}

\begin{corollary}\label{Hnilp} Let $k \to k^{\mathbb Z_p} \to H \to kF \to k$ be an abelian exact sequence such  that   $|G(H)| = p$. Then $H$ is nilpotent. \end{corollary}

\begin{proof} It follows from Proposition \ref{Fnilp}, in view of Remark \ref{central-nilpotent}. \end{proof}

\begin{remark}\label{i-k} In view of  \cite[Theorem IX.8 (iii)]{IK},  if $H$ is a Kac algebra with $\cd (H^*) = \{1, p\}$ and $|G(H)| = p$, then $H$ is a central abelian extension associated to an action of the cyclic group of order $p$ on a nilpotent group. It follows from Corollary \ref{Hnilp} that $H$ is a nilpotent Hopf algebra. \end{remark}

\begin{remark} Note that the (dual) assumption that $\cd (H) = \{ 1, p \}$ does not imply that $H$ is nilpotent in general.
For example, take $H$ to be the group algebra of a nonabelian
semidirect product $F \rtimes \mathbb Z_p$, where $F$ is an
abelian group such that $(|F|, p) = 1$.

\medbreak On the other hand, the assumption on $|G(H)|$ in
Corollary \ref{Hnilp} and Proposition \ref{Fnilp} is essential.
Namely, for all prime number $p$, there exist semisimple Hopf
algebras $H$ with $\cd(H^*) = \{  1, p\}$ and such that $H$ is
\emph{not} nilpotent.

To see an example, consider  a group $F$ with an automorphism of order $p$ and suppose $F$ is not nilpotent  (take for instance $F = \mathbb S_n$, a symmetric group, such that $n > 6$ is sufficiently large). Consider the corresponding action of $\mathbb Z_p$ on $F$ by group automorphisms and let $G = F \rtimes \mathbb Z_p$ be the semidirect product.

Then there is an associated (split) abelian exact sequence $k \to k^{\mathbb Z_p} \to H \to kF \to k$, such that $H$ is not commutative and not cocommutative. Moreover, in view of Remark \ref{irr-abext}, $\cd (H^*) = \{ 1, p\}$. But, by Remark \ref{central-nilpotent}, $H$ is not nilpotent, because $F$ is not nilpotent by assumption.
\end{remark}

\subsection{Reduction to abelian extensions from character degrees}

In this subsection we consider the case where $\cd (H) = \{1,p\}$ for some prime $p$ and $|G(H^*)| = p$. We treat the problem of deducing an abelian extension like \eqref{abext-1p} from this assumption.

It is known, for instance, that if $p = 2$, then the assumption implies that $H$ is cocommutative \cite[Proposition 6.8]{BN}, \cite[Corollary IX.9]{IK}.

\begin{lemma}\label{cocommutative} Suppose that $\cd (H^*) = \{1,p\}$ for some prime $p$.   Then $H/(kG(H))^+H$ is a cocommutative coalgebra. \end{lemma}

\begin{proof}  Let $\chi$ be an irreducible character of degree $p$. We have that $$\chi\chi^* = \underset{g\in G[\chi]}{\sum} g + \underset{\deg \lambda = p}{\sum} \lambda.$$  So $p||G[\chi]|$. Therefore $|G[\chi]|$ is either $p = \deg \chi$ or $p^2$, because it divides $(\deg \chi)^2$.

Moreover,  since $\chi = g \chi$ for all $g\in G[\chi]$, we have $G[\chi]C = C$, where $C$ is the simple subcoalgebra of $H$ containing $\chi$. Then it follows from \cite[Remark 3.2.7]{ssld} that $C/(kG[\chi])^+C$ is a cocommutative coalgebra (indeed, $|G[\chi]|$ is either $p = \deg \chi$ or $p^2$, but in the last case, $C/(kG[\chi])^+C$ is one-dimensional, hence also cocommutative). Then $H/(kG(H))^+H$ is a cocommutative coalgebra, by \cite[Corollary 3.3.2]{ssld}. \end{proof}

\subsection{Results for the type $(1, p; p, n)$} Let $p$ be a prime number.
Along this subsection $H$ will be a semisimple Hopf algebra such that $\cd(H) = \{ 1, p \}$ and $|G(H^*)| = p$.
So that $H$ is of type $(1, p; p, n)$ as an algebra.

\begin{proposition}\label{dual-nilp} Suppose that $p$ divides $|G(H)|$.
Then $G(H^*) \subseteq Z(H^*)$ and  $H^*$ is nilpotent. \end{proposition}

\begin{proof} By assumption, there is a subgroup $G$ of $G(H)$
with $|G| = p$ (i.e. $G\simeq \mathbb{Z}_p$) and the Hopf algebra
inclusion  $kG \rightarrow H$ induces the following sequence:
$$kG(H^*)\xrightarrow{\emph{i}} H^*\xrightarrow{\pi} kG,$$ with $\pi$
surjective. By \cite[Lemma 4.1.9]{ssld}, setting $A = kG(H^*)$ and
$B = kG$, we have that $\pi\circ \emph{i}:kG(H^*)\rightarrow kG$ is
an isomorphism and $H^*\simeq R \# kG(H^*)\simeq R \# \mathbb{Z}_p$ is
a biproduct, where $R\doteq (H^*)^{\co\pi}$ is a semisimple braided
Hopf algebra over $\mathbb{Z}_p$. The coalgebra $R$ is
cocommutative, by Lemma \eqref{cocommutative}, because $R\simeq
H^*/H^*kG(H^*)^+$ as coalgebras.  Since $p \nmid 1 + np = \dim R$ then
by \cite[Proposition 7.2]{So}, $R$ is trivial. Therefore, by
\cite[Proposition 4.6.1]{ssld}, $H^*$ fits into an abelian
\emph{central} exact sequence $$k\rightarrow k\mathbb{Z}_p \rightarrow
H^*\rightarrow R\rightarrow k.$$

Now, since the extension is abelian, there is a group $F$
such that $R\simeq kF$. It follows from Corollary \ref{Hnilp} that $H^*$ is nilpotent.
\end{proof}

\begin{proposition}\label{qt-1p} Suppose $H$ is quasitriangular. Then $G(H^*) \subseteq Z(H^*)$ and $H^*$ is nilpotent.
\end{proposition}

\begin{proof} Consider the Drinfeld double $D(H)$. Since $H$ is quasitriangular, $G(H^*) \simeq \mathbb Z_p$ is isomorphic to a subgroup of $G(D(H)^*)$.
Then $G(D(H)^*)$ has an element $g\# f$ of order $p$. We have that $G(D(H)^*) \simeq G(D(H)) \cap Z(D(H)) \subseteq G(D(H)) = G(H^*) \times G(H)$; see Subsection \ref{quasitriangular}.

In particular, the element $f \# g \in G(D(H)) \cap Z(D(H))$ is of order $p$.
If $g$ is of order $p$, then the proposition follows from Proposition \ref{dual-nilp}.
Thus we may assume that $g = 1$. Then $f \in G(H^*) \cap Z(H^*)$ is of order $p$, implying that $G(H^*) \subseteq Z(H^*)$.

\medbreak Therefore $H^*$ fits into an abelian central exact
sequence $k \to k^{\mathbb Z_p} \to H^* \to kF \to k$, where $F$
is a finite group such that $kF \simeq H^*/H^*(k^{\mathbb
Z_p})^+$, by Lemma \ref{cocommutative}. In view of the assumption
on the algebra structure of $H$, Corollary \ref{Hnilp} implies
that $H^*$ is nilpotent, as claimed. \end{proof}

\subsection{Results for the type $(1, p; p, 1)$}  We next discuss the case where $H$ is of type $(1, p; p, 1)$ as an algebra (not necessarily quasitriangular). In particular, $\dim H = p(p+1)$ is even.

Notice that under this assumption, the category $\Rep H$ is a \emph{near-group category} with fusion rule given by the group $G = G(H^*)\simeq \mathbb{Z}_p$ and the integer $\kappa$ \cite{Siehler}.

Let $\chi$ be the irreducible character of degree $p$. It follows that $\chi = \chi^*$ and $\chi g = \chi = g\chi$. Then $$\chi^2 = \underset{g\in G(H^*)}{\sum} g + \kappa\chi.$$
Taking degrees in the equation above we obtain $p^2 = p+\kappa p$,
which means that  $\kappa = p-1$.

We shall use the following proposition. A more general statement will be proved in Theorem \ref{solv-neargp}.

\begin{proposition}\label{caso1sol} Suppose $H$ is of type $(1, p; p, 1)$ as an algebra. Then one of the following holds:

\emph{(i)} $p = 2$ and $H \simeq k\mathbb S_3$, or

\emph{(ii)} $p = 2^{\alpha}-1$\footnote{Such a prime number is called a
\emph{Mersenne prime}, in particular $\alpha$ must be prime.}, and $\dim H = 2^{\alpha}p$.

In particular, $H$ is solvable. \end{proposition}

\begin{proof}   By \cite[Theorem 1.2]{Siehler},
it follows that $G(H^*)\simeq \mathbb{Z}_{q^\alpha -1}$, for some
prime $q$ and $\alpha \geq 1$. Therefore $p = q^\alpha -1$. If $q
> 2$, then  $p = 2$, which implies $H\simeq k\mathbb S_3$ is
cocommutative. If $q = 2$, then $p$ has the particular expression
$p = 2^{\alpha} - 1$.

Hence  $\dim H$ equals $6$ or $p(p+1) = 2^{\alpha}p$. By Burnside's theorem for fusion categories \cite[Theorem 1.6]{ENO2}, $H$ is solvable. \end{proof}

\begin{remark}\label{rmk-seitz} Let $p$ be a prime number such that $p = 2^{\alpha}-1$, as in Proposition \ref{caso1sol}. Consider the affine group $N$ of the field $\mathbb F_{2^\alpha}$, that is, $N$ is the semidirect product $\mathbb F_{2^\alpha} \rtimes \mathbb F_{2^\alpha}^{\times}$ with respect to the natural action of $\mathbb F_{2^\alpha}^{\times}$ on $\mathbb F_{2^\alpha}$. Then the group $N$ has the prescribed algebra type (see \cite[Subsection 4.1]{Siehler}).

Furthermore, suppose $p$ is (any) prime number, and $N$ is a group whose group algebra has algebra type $(1, p; p, 1)$. Then $N$ has order $p(p+1)$ and it follows from the main result of \cite{seitz} that either $p = 2$ and $N \simeq \mathbb S_3$ or $p = 2^{\alpha}-1$, $\alpha > 1$, and $N \simeq \mathbb F_{2^\alpha} \rtimes \mathbb F_{2^\alpha}^{\times}$.
\end{remark}


\begin{proposition}\label{type1pp1}
Let $H$ be a semisimple Hopf algebra of type $(1, p; p, 1)$ as an algebra. Then $G(H^*) \subseteq Z(H^*)$ and $H^*$ is nilpotent.
\end{proposition}

\begin{proof}
We have just proved in Proposition \ref{caso1sol} that under this hypothesis $H$ is solvable. Since $\Rep D(H)\simeq Z(\Rep H)$, then $D(H)$ is also solvable \cite[Proposition 4.5 (i)]{ENO2}.

By \cite[Proposition 4.5 (iv)]{ENO2}, $D(H)$ has nontrivial representations of dimension $1$, that is, $|G(D(H)^*)|\neq 1$. We have that $G(D(H)^*) \simeq G(D(H)) \cap Z(D(H)) \subseteq G(D(H)) = G(H^*) \times G(H)$; see Subsection \ref{quasitriangular}.

We next argue as in the proof of Proposition \ref{qt-1p}. Consider
an element $1\neq  f \# g \in G(D(H)) \cap Z(D(H))$. If $f = 1$,
then $1 \neq g \in Z(H) \cap G(H)$. Therefore, $H^*$ fits into a
cocentral extension $k \to K \to H^* \to k^{\langle g \rangle} \to
k$, where $K$ is a \emph{proper} normal Hopf subalgebra. The
assumption on the algebra structure of $H$ implies that $K =
kG(H^*)$. Thus $kG(H^*)$ is normal in $H^*$, and the extension is
abelian, by Lemma \ref{cocommutative}. The proposition follows in
this case from Proposition \ref{Fnilp} (i) and Corollary \ref{Hnilp}.

Thus we may assume that $f \neq 1$. In particular, $f$ has order $p$.

\medbreak If $|f| = |g| = p = |G(H^*)|$, we have that $p| |G(H)|$. Then $G(H^*)\subseteq Z(H^*)$ and $H^*$ is nilpotent, by Proposition \ref{dual-nilp}.

Otherwise, take $|g| = n$, with $p\neq n$. If $f^n = 1$, then $p$ divides $n$ and thus $p$ divides $|G(H)|$. As before, we are done by Proposition \ref{dual-nilp}.

If $f^n \neq 1$, then $f^n\# 1 = (f^n\# g^n) = (f\# g)^n\in Z(D(H))$, which implies that $f^n\neq 1$ is central in $H^*$ and thus $G(H^*)\subseteq Z(H^*)$.

\medbreak Therefore $H^*$ fits into an abelian central exact sequence $k \to k^{\mathbb Z_p} \to H^* \to kF \to k$, where $F$ is a finite group such that $kF \simeq H^*/H^*(k^{\mathbb Z_p})^+$, by Lemma \ref{cocommutative}.
In view of the assumption on the algebra structure of $H$, Corollary \ref{Hnilp} implies that $H^*$ is nilpotent, as claimed.
\end{proof}

\begin{theorem}\label{twist-aff}
Let $H$ be a semisimple Hopf algebra of type $(1, p, p, 1)$ as an algebra.
Then either $p = 2$ and $H \simeq k\mathbb S_3$, or $H$ is isomorphic to a twisting of the group algebra $kN$, where $p = 2^{\alpha} -1$, $\alpha > 1$, and $N$ is the affine group of the field $\mathbb F_{2^\alpha}$.
\end{theorem}

\begin{proof} If $p  = 2$, then $\dim H = 6$ and the result follows from \cite{masuoka-6-8}.
So suppose that $p$ is odd.
By Propositions \ref{type1pp1} and \ref{caso1sol}, $H^*$ fits into an abelian central exact sequence $k \to k^{\mathbb Z_p} \to H^* \to kF \to k$, where $F$ is a finite group of order $p + 1 = 2^{\alpha}$.
Then the action $\vartriangleleft: \mathbb Z_p \times F \to \mathbb Z_p$ is trivial, while the  action $\vartriangleright: \mathbb Z_p \times F \to F$ is determined by an automorphism $\varphi\in \Aut F$ of order $p = 2^\alpha - 1$.

\medbreak We first claim that the group $F$ must be abelian. By a result of P. Hall \cite[5.3.3]{robinson}, since $F$ is a $2$-group, the order of $\Aut F$ divides the number $n 2^{(\alpha - r)r}$, where $n = |\GL (r, 2)|$ and $2^r$ equals the index in $F$ of the Frattini subgroup $\Frat (F)$  (which is defined as the intersection of all the maximal subgroups of $F$ \cite[pp. 135]{robinson}). In particular, we have $r \leq \alpha$.

Since the order of $\varphi$ divides the order of $\Aut F$ and $|\GL (r, 2)| = (2^r - 1)(2^r - 2)\ldots (2^r - 2^{r-1})$, it follows that the prime $p = 2^\alpha - 1$ divides $2^r - 1$, which means that $r = \alpha$ and, therefore, $\Frat (F) = 1$.

Since $F$ is nilpotent (because it is a $2$-group), a result of Wielandt \cite[5.2.16]{robinson} implies that $[F, F]$, the commutator subgroup of $F$, is a subgroup of the Frattini subgroup $\Frat (F)$. As we have just shown, we have $\Frat (F) = 1$ in this case. Thus $[F, F] = 1$ and therefore $F$ is abelian, as claimed.

\medbreak Consider the split extension $B_0 = k^{\mathbb Z_p} \# kF$ associated to the matched pair $(\mathbb Z_p, F)$. Since $F$ is abelian, $B_0$ (being a central extension) is commutative. This means that $B_0$ isomorphic to $k^N$, where $N = F \rtimes \mathbb Z_p$.

Notice that $|F| = 2^\alpha$ is relatively prime to $p$. It follows from \cite[Proposition 5.22]{exp-dpr} and \cite[Proposition 3.1]{ma-ext-cohom} that $H^*$ is obtained from the split extension $B_0 = k^{\mathbb Z_p} \# kF \simeq k^N$ by twisting the multiplication. Indeed, the element representing the class of $H^*$ in the group $\Opext(kF, k^{\mathbb Z_p})$ is the image of an element of $H^2(F, k^{\times})$ under the map $H^2(F, k^{\times}) \oplus H^2(\mathbb Z_p, k^{\times}) \simeq H^2(F, k^{\times}) \to \Opext(kF, k^{\mathbb Z_p})$ in the Kac exact sequence \cite[Theorem 1.10]{ma-ext-cohom}. Then the claim follows from \cite[Proposition 3.1]{ma-ext-cohom}.  Dualizing, we get that $H$ is a twisting of the group algebra of the gr oup $N$.

Finally, the assumption on the algebra structure of $H$ implies that $N$ is one of the claimed groups. See Remark \ref{rmk-seitz}.
\end{proof}


\begin{corollary}
Let $H$ be a semisimple Hopf algebra of type $(1, p, p, 1)$ as an algebra. Then $\Rep H \simeq \Rep N$, where $N = \mathbb S_3$ or  $N$ is the affine group of the field $\mathbb F_{2^\alpha}$, for some $\alpha > 1$.
\end{corollary}

\section{Solvability}\label{solvability}

Recall from \cite{ENO2} that a fusion category $\C$ is called
\emph{weakly group-theoretical} if it is Morita equivalent to a
nilpotent fusion category. If, furthermore, $\mathcal C$ is Morita
equivalent to a cyclically nilpotent fusion category, then $\C$ is
called \emph{solvable}.

In other words, $\mathcal C$ is weakly group-theoretical (solvable) if there exists an indecomposable algebra $A$ in $\C$ such that the category ${}_A\C_A$ of $A$-bimodules in $\C$ is a  (cyclically)  nilpotent fusion category.

\medbreak Note that a group-theoretical fusion category is weakly group-theoretical.

On the other hand, the condition on  $\C$ being solvable is equivalent to the existence of a sequence of fusion categories $$\C_0 = \vect_k,
\C_1, \dots , \C_n = \C,$$   such that $\C_i$ is obtained from
$\C_{i-1}$ either by a $G_i$-equivariantization or as a
$G_i$-extension, where $G_1, \dots , G_n$ are cyclic groups
of prime order. See \cite[Proposition 4.4]{ENO2}.

\medbreak If $G$ is  a finite group and $\omega \in H^3(G,
k^{\times})$, we have that the categories $\C(G, \omega)$ and
$\Rep G$ are solvable if and only if $G$ is solvable.

\medbreak Let us call a semisimple Hopf algebra $H$ \emph{weakly
group-theoretical} or \emph{sol\-vable}, if the category $\Rep H$
is weakly group-theoretical or solvable, respectively.

\subsection{Solvability of an abelian extension}
By \cite[Proposition 4.5 (i)]{ENO2},  solvability of a fusion category is preserved under Morita equivalence.
Therefore, a group-theoretical fusion category $\C(G, \omega, F, \alpha)$ is solvable if and only if the group $G$ is solvable.

\begin{remark} As a consequence of the Feit-Thompson theorem \cite{feit-thompson}, we
get that if the order of $G$ is odd, then $\C(G, \omega, F,
\alpha)$ is solvable. This fact generalizes
to weakly group-theoretical fusion categories; see Proposition \ref{odd-wgt} below.
\end{remark}

\medbreak This implies the following characterization of the solvability of an abelian extension:

\begin{corollary}\label{G-solv} Let $H$ is a semisimple Hopf algebra
fitting into an abelian exact sequence \eqref{abelian-sec},  then
$H$ is solvable if and only if $G = F\bowtie \Gamma$ is solvable. \end{corollary}

In particular, if $H$ is solvable, then $F$ and $\Gamma$ are solvable.

\medbreak A result of Wielandt \cite{wielandt} implies that if the
groups $\Gamma$ and $F$ are nilpotent, then $G$ is solvable. As a
consequence, we get the following:

\begin{corollary}\label{ab-nilp} Suppose $\Gamma$ and $F$ are nilpotent. Then $H$ is solvable. \end{corollary}

Then, for instance, the abelian extensions in Proposition
\ref{Fnilp} are solvable.

\medbreak Combining Corollary \ref{ab-nilp} with Lemma \ref{Fnilp-gt} and Remark \ref{rmk-gt}, we get:

\begin{corollary}\label{solv-gt} Let $\C = \C(G, \omega, \mathbb Z_p, \alpha)$. Assume that   $|G(\C)| = p$. Then  $\C$ is solvable.
 \end{corollary}

\section{Solvability from character degrees}\label{six}

Let $p$ be a prime number. We study in this section fusion categories $\C$ such that $\cd(\C) = \{ 1, p \}$.

It is known that if  $G$ is a finite group, then this
assumption implies that the group $G$, and thus the category $\Rep G$, are solvable \cite{isaacs}.

\begin{remark}\label{tensor-prod} If $H$ is any semisimple Hopf algebra such that $\cd(H) = \{ 1, p\}$ and $G$ is any finite group, then the tensor product Hopf algebra
$A = H \otimes k^G$ also satisfies that $\cd(A) = \{ 1, p\}$
(since the irreducible modules of $A$ are tensor products of
irreducible modules of $H$ and $k^G$).

But $A$ is not solvable
unless $G$ is solvable; indeed, $k^G$ is a Hopf subalgebra as well
as a quotient Hopf algebra of $A$.
\end{remark}

\medbreak Our aim in this section is to prove some structural results on $\C$, regarding solvability, under additional restrictions.

The following theorem generalizes Proposition \ref{caso1sol}.

\begin{theorem}\label{solv-neargp} Let $\C$ be a near-group fusion category such that $\cd(\C) = \{ 1, p\}$. Then $\C$ is solvable. \end{theorem}

\begin{proof} In the notation of \cite{Siehler}, let the fusion rules of $\C$ be given by the pair $(G, \kappa)$, where $G$ is the group of invertible objects of $\C$ and $\kappa$ is a  nonnegative integer.
Then $\Irr(\C) = G \cup \{ m\}$, with the relation \begin{equation}\label{m2}m^2 = \sum_{g \in G}g + \kappa m.\end{equation}

The assumption on $\cd(\C)$ implies that $\FPdim m = p$. Hence $\FPdim \C = |G| + p^2$, and since $|G| = |G(\C)|$ divides $\FPdim \C$, we get that $|G| = p$ or $p^2$. (Note that taking Frobenius-Perron dimensions in \eqref{m2}, we get that  $G \neq 1$.)


\medbreak If $|G| = p^2$, then $\kappa = 0$ and $\C$ is a Tambara-Yamagami category \cite{TY}. Furthermore, $\C$ is a $\mathbb Z_2$-extension of a pointed category $\C(G, \omega)$. Then $\C$ is solvable in this case, by \cite[Proposition 4.5 (i)]{ENO2}.

\medbreak Suppose that $|G| = p$. Then  $\kappa = p-1$. As in the proof of Proposition \ref{caso1sol}, using \cite[Theorem 1.2]{Siehler}, we get that $\FPdim \C = p(p+1)$ equals $6$ or $p2^{\alpha}$. Then $\C$ is solvable, by \cite[Theorem 1.6]{ENO2}. \end{proof}

Our next result is the following theorem, for $\C = \Rep H$, which is a consequence of Proposition \ref{qt-1p}.
A stronger version of this result will be given in Subsection \ref{braided-fusion}, under additional dimension restrictions.

\begin{theorem} Suppose $H$ is of type $(1, p; p, n)$ as an algebra. Assume in addition that $H$ is quasitriangular. Then $H$ is solvable.
\end{theorem}

\begin{proof} We have shown in Proposition \ref{qt-1p} that $H^*$ is nilpotent. Moreover, by Lemma \ref{cocommutative}, $H$ fits into an abelian cocentral exact sequence
$k \to k^F \to H \to k\mathbb Z_p \to k$, where $F$ is a nilpotent group.
Therefore, $H$ is solvable, by Corollary \ref{ab-nilp}.  \end{proof}

\medbreak In the remaining of this section, we restrict ourselves to the case where $\C = \Rep H$ for a semisimple Hopf algebra $H$.

\subsection{The case $p = 2$}

Let $H$ be a semisimple Hopf algebra such that $\cd(H)
\subseteq \{ 1, 2\}$. By \cite[Theorem 6.4]{BN}, one of the following
possibilities holds:
\begin{enumerate}\item[(i)] There is a cocentral abelian exact sequence $k \to k^F \to H \to
k\Gamma \to k$, where $F$ is a  finite group and $\Gamma \simeq
{\mathbb Z_2}^n$, $n \geq 1$, or
\item[(ii)] There is a central exact sequence $k \to k^U \to H \to
B \to k$, where $B = H_{\ad}$ is a proper Hopf algebra quotient, and
$U = U(\Rep H)$ is the universal grading group of the category
of finite dimensional $H$-modules.
\end{enumerate}
In particular, if $H = H_{\ad}$, then $H$ satisfies (i).

\medbreak As a consequence of this result we have:

\begin{theorem}\label{2-wgt} Let $H$ be a semisimple Hopf algebra such that $\cd(H)
\subseteq \{ 1, 2\}$. Then $H$ is weakly group-theoretical.

Moreover, if  $H = H_{\ad}$, then $H$ is group-theoretical. \end{theorem}

\begin{proof} The assumption implies that $H$ satisfies (i) or (ii) above. If $H$ satisfies (i), then  $H$ is group-theoretical, by Remark \ref{H-group-theoretical}.

Otherwise, $H$ satisfies (ii), and then the category $\Rep H$ is a $U$-extension of $\Rep B$, in view of Proposition \ref{grading-hopf}.
By an inductive argument, we may assume that $B$ is weakly group-theoretical (note that $\cd(B)
\subseteq \{ 1, 2\}$). Therefore so is $H$, by \cite[Proposition 4.1]{ENO2}.
\end{proof}

We next discuss conditions that guarantee the solvability of $H$. The following result is proved in \cite{BN}.

\begin{proposition} \cite[Proposition 6.8]{BN}. Suppose $H$ is of type $(1, 2; 2, n)$ as an algebra. Then $H$ is
cocommutative. \end{proposition}

The proposition implies that such a Hopf algebra $H$  is isomorphic to a group algebra $kG$ for some finite group $G$. By the assumption on the algebra structure of $H$, the group $G$, and then also $H$, are solvable.

The next lemma gives a sufficient condition for $H$ to be solvable.

\begin{lemma}\label{suff-solv} Suppose $\cd(H) \subseteq \{ 1, 2\}$ and $H = H_{\ad}$. Then $H$ is solvable if and only if the group $F$ in (i) is solvable.
\end{lemma}

\begin{proof} Since $H = H_{\ad}$, then $H$ satisfies (i). Therefore $H$ is solvable if and only if the relevant factorizable group $G = F\bowtie \Gamma$ is solvable, by Corollary \ref{G-solv}. Also, since the sequence (i) is cocentral, then $G$ is a semidirect product: $G = F\rtimes \Gamma$. This proves the lemma. \end{proof}

\begin{remark} Suppose that $H$ has a faithful irreducible character $\chi$ of degree $2$, such that $\chi\chi^* = \chi^*\chi$. Then it follows from  \cite[Theorem 3.5]{BN} that  $H$ fits into a central abelian exact sequence  $k\to k^{\mathbb Z_m} \to H \to  kT \to k$,
for some polyhedral group $T$ of even order and for some $m\geq 1$. In particular, since $\cd(H) = \{ 1, 2 \}$, then $T$ is necessarily cyclic or dihedral (see, for instance, \cite[pp. 10]{BN} for a description of the polyhedral groups and their character degrees). Therefore $H$ is solvable in this case.

The assumption on $\chi$ is satisfied in the case where $H$ is quasitriangular; so that the conclusion holds in this case. We shall show in the next subsection that every quasitriangular semisimple Hopf algebra with $\cd (H) \subseteq \{ 1, 2\}$ is also solvable.
\end{remark}

We next prove some lemmas that will be useful in the next subsection.

\begin{lemma}\label{sub-qt} Suppose $\cd(H) \subseteq \{ 1, 2\}$ and let $K$ be a Hopf subalgebra or quotient Hopf algebra of $H$. Then $\cd(K) \subseteq \{ 1, 2\}$.
\end{lemma}

\begin{proof} We only need to show the claim when $K \subseteq H$ is a Hopf subalgebra.
In this case, the statement follows from surjectivity of the
restriction functor $\Rep H \to \Rep K$. \end{proof}

The lemma has the following immediate consequence:

\begin{corollary}\label{2-gdeh} If $\cd(H) \subseteq \{ 1, 2\}$, then the group $G(H)$ is solvable. \end{corollary}

\begin{lemma}\label{h-h*} Suppose $\cd(H), \cd(H^*) \subseteq \{ 1, 2\}$. Then
$H$ is solvable. \end{lemma}

\begin{proof} By induction on the dimension of $H$.

Consider  the universal grading group $U$ of the category $\Rep
H$. Then $H^* \to kU$ is a quotient Hopf algebra and therefore
$\cd(U)\subseteq \{ 1, 2\}$, by Lemma \ref{sub-qt}. This implies
that the group $U$ is solvable.

\medbreak Suppose first $H_{\ad} \neq H$.  In view of Lemma \ref{sub-qt},
we also have $\cd(H_{\ad})$, $\cd(H_{\ad}^*) \subseteq \{ 1, 2\}$.
By the inductive assumption $H_{\ad}$ is solvable. By \cite[Proposition 4.5 (i)]{ENO2}, $H$ is
solvable, since $\Rep H$ is a $U$-extension of $\Rep H_{\ad}$.

\medbreak It remains to consider the case where $H_{\ad} = H$. As pointed
out at the beginning of this subsection,  it follows from
\cite[Theorem 6.4]{BN} that in this case $H$ satisfies condition
(i), that is, $H$ fits into a cocentral abelian exact sequence $k
\to k^F \to H \to k\Gamma \to k$, with $|\Gamma| > 1$ and $\Gamma$
abelian.

In particular, $k^{\Gamma} \subseteq H^*$ is a nontrivial central
Hopf subalgebra, implying that $H^* \neq  H^*_{\ad}$. The
inductive assumption implies, as before, that $H^*_{\ad}$ and thus
also $H^*$ is solvable. Then $H$ also is. This finishes the proof
of the lemma.\end{proof}

\subsection{The quasitriangular case}

We shall assume in this subsection that $H$ is quasitriangular. Let $R \in H \otimes H$ be an $R$-matrix. We keep the notation in Subsection \ref{quasitriangular}.

\begin{remark} Notice that, since the category $\Rep H$ is braided, then the
universal grading group $U = U(\Rep H)$ is abelian (and in
particular, sol\-vable). \end{remark}

\medbreak The following is the main result of this subsection.

\begin{theorem}\label{solv-2} Let $H$ be a  quasitriangular semisimple Hopf algebra such that $\cd(H)
\subseteq \{ 1, 2\}$. Then $H$ is solvable. \end{theorem}

\begin{proof} If $\cd(H) = \{ 1\}$, then $H$ is commutative and, because it is quasitriangular, isomorphic to the group algebra of an abelian
group. Hence we may assume that $\cd(H) = \{ 1, 2\}$.

Consider the Hopf subalgebras $H_+, H_- \subseteq H$. By Lemma
\ref{sub-qt}, we have $\cd(H_+)$, $\cd(H_-) \subseteq \{ 1, 2 \}$.
Then $\cd(H_-), \cd(H_-^*) \subseteq \{ 1, 2 \}$, since $(H_-^*)^{\cop}
\simeq H_+$.

By Lemma \ref{h-h*}, $H_-$ is solvable. Therefore the Drinfeld
double $D(H_-)$ and its homomorphic image $H_R$ are also solvable.

We may thus assume that $H_R \subsetneq H$.

Observe that, being a quotient of $H$, $H_{\ad}$ is also
quasitriangular and satisfies $\cd(H_{\ad}) \subseteq \{ 1, 2 \}$.
Hence, by induction, we may also assume that $H = H_{\ad}$, and in
particular, $G(H) \cap Z(H) = 1$. Indeed, $\Rep H$ is a
$U$-extension of $\Rep H_{\ad}$ and the group $U$ is abelian, as
pointed out before.

\medbreak Therefore $H$ fits into a cocentral abelian exact
sequence $k \to k^F \to H \to k\Gamma \to k$, where $1 \neq
\Gamma$ is elementary abelian of exponent $2$.

In view of Lemma \ref{suff-solv}, it will be enough to show that the group $F$ is solvable.

\medbreak We have $\widehat \Gamma \subseteq G(H^*)\cap Z(H^*)$.
By \cite[Proposition 3]{radford}, $f_{R_{21}}(G(H^*)\cap Z(H^*))
\subseteq G(H) \cap Z(H)$. Hence we may assume that
$f_{R_{21}}|_{\widehat\Gamma} = 1$ and similarly
$f_R|_{\widehat\Gamma} = 1$. Thus $f_R$, $f_{R_{21}}$
factorize through the quotient $H^*/H^*(k\widehat \Gamma)^+ \simeq
kF$.

Therefore $H_+ = f_R(H^*)$ and $H_- = f_{R_{21}}(H^*)$ are
cocommutative. (Then they are also commutative, since $H_+ \simeq
H_-^{*\cop}$.) In particular, $H_R = H_+H_-$ is cocommutative.
Hence $\Phi_R(H^*) \subseteq H_R \subseteq kG(H)$.

\medbreak By \cite[Theorem 4.11]{qt-quotient}, $K = \Phi_R(H^*)$
is a commutative (and cocommutative) normal Hopf subalgebra, which is necessarily
solvable, since $H_R$ is. In addition,  $\Phi_R(H^*) \simeq kT$,
where $T \subseteq G(H)$ is an abelian subgroup (\cite[Example 2.1]{qt-quotient}), and there is an
exact sequence of Hopf algebras
\begin{equation*}k \to kT \to H \overset{\pi}\to \overline H \to k, \end{equation*} where
$\overline H$ is a certain (canonical) triangular Hopf algebra.

Since $\overline H$ is triangular, then $\overline H \simeq
(kL)^J$, is a twisting of the group algebra of some finite group
$L$. Because $\cd(L) = \cd(\overline H) \subseteq \{ 1, 2\}$, $L$
must be solvable. Hence $\overline H$ is solvable, since $\Rep
\overline H \simeq \Rep L$.

\medbreak The map $\pi: H \to \overline H$ induces, by restriction to the Hopf subalgebra $k^F \subseteq H$, an exact sequence
\begin{equation*}k \to kT\cap k^F \to k^F \overset{\pi|_{k^F}}\to \pi(k^F) \to k. \end{equation*}
We have $kT\cap k^F = k^{\overline F}$ and $\pi(k^F) = k^S$, where
$\overline F$ and $S$ are a quotient and a subgroup of $F$,
respectively, in such a way that the exact sequence above
corresponds to an exact sequence of groups
\begin{equation*}1 \to S \to F \to \overline F \to 1. \end{equation*}
Now, $\overline F$ is abelian, because $k^{\overline F} = kT\cap
k^F$ is cocommutative, and $S$ is solvable, because $k^S$ is a
Hopf subalgebra of $\overline H$. Therefore $F$ is solvable. This
implies that $H$ is solvable and finishes the proof of the
theorem. \end{proof}

\section{Odd dimensional fusion categories}\label{seven}

Along this section, $p$ will be a prime number.  Let $\mathcal C$
be a fusion category over $k$. Recall that the set of irreducible degrees of $\C$ was defined as  \begin{equation*}\cd(\mathcal C) =  \{
\FPdim x|\, x \in \Irr \C\}.
\end{equation*}

The fusion categories that we shall consider in this section are all \emph{integral}, that is, the Frobenius-Perron dimensions of objects of $\C$ are (natural) integers. By \cite[Theorem 8.33]{ENO},  $\C$ is isomorphic to the category of representations of some finite dimensional semisimple quasi-Hopf algebra.

\subsection{Odd dimensional weakly group-theoretical fusion categories} The following result is a consequence of the Feit-Thompson theorem \cite{feit-thompson}.

\begin{proposition}\label{odd-wgt} Let $\mathcal C$ be a weakly group-theoretical fusion category and assume that $\FPdim \mathcal C$ is an odd integer. Then $\mathcal C$ is solvable. \end{proposition}

Note that since  $\FPdim \mathcal C$ is an odd integer, the fusion category $\C$ is  integral. See \cite[Corollary 2.22]{DGNO1}.


\begin{proof} By definition, $\C$ is Morita equivalent to a nilpotent fusion category.   Then, by \cite[Proposition 4.5 (i)]{ENO2}, it will be enough to show that a \emph{nilpotent} fusion category of odd Frobenius-Perron dimension is solvable. So, assume that $\C$ is nilpotent, so that $\C$ is a $G$-extension of a fusion subcategory $\widetilde C$, with $|G| > 1$. In particular, $\FPdim \C = |G| \FPdim \widetilde \C$. Hence the order of $G$ and $\FPdim \widetilde \C$ are odd and $\FPdim \widetilde \C < \FPdim \C$.  The theorem follows by induction, since  by the Feit-Thompson theorem, $G$ is solvable. See \cite[Proposition 4.5 (i)]{ENO2}.\end{proof}



\subsection{Braided fusion categories}\label{braided-fusion}

We shall need the following lemma whose proof is contained in the
proof of \cite[Proposition 6.2 (i)]{ENO2}.  We include a sketch of
the argument for the sake of completeness.

\begin{lemma}\label{cd-cg} Let $\C$ be a fusion category and let $G$ be a finite group acting on $\C$ by
tensor autoequivalences. Assume  $\cd(\C^G) \subseteq \{ p^m:\, m
\geq 0\}$, where $p$ is a prime number. Then $\cd(\C) \subseteq \{
p^m:\, m \geq 0\}$.
\end{lemma}

\begin{proof} Regard $\C$ as an indecomposable module category over itself via tensor product,
and similarly for $\C^G$.  Let $Y$ be a  simple object of $\C$.
Since the forgetful functor $F: \C^G \to \C$ is surjective, $Y$ is
a simple constituent of $F(X)$, for some  simple object $X$ of
$\C^G$.

Since $F$ is a tensor functor, we have $\FPdim X = \FPdim F(X)$. By Formula (7) in \cite[Proof of Proposition 6.2]{ENO2},
\begin{equation}\FPdim (X) = \deg(\pi) [G: G_Y] \FPdim Y,\end{equation}
where $G_Y \subseteq G$ is the stabilizer of $Y$ and $\pi$ is an
irreducible representation of $G_Y$ associated to $X$. Therefore
$\FPdim Y$ divides $\FPdim X$.

The assumption on $\C^G$ implies that $\FPdim X$ is a power of
$p$. Then so is $\FPdim Y$. This proves the lemma.
\end{proof}

\begin{theorem}\label{braided-odd} Let $\mathcal C$ be braided fusion category
such that $\cd(\mathcal C) \subseteq \{ p^m:\, m \geq 0\}$, where
$p$ is a prime number. Assume that $\FPdim \C$ is odd. Then
$\mathcal C$ is solvable.
\end{theorem}

\begin{proof} By induction on $\FPdim \C$.
(Notice that the Frobenius-Perron dimension of a fusion
subcategory of $\C$  divides the dimension of $\C$
\cite[Proposition 8.15]{ENO}, and the same is true for the
Frobenius-Perron dimension of a fusion category $\mathcal D$ such
that there exists a surjective tensor functor $\C \to \mathcal D$
\cite[Corollary 8.11]{ENO}. Thus these fusion categories are
odd-dimensional as well.) If $\cd(\mathcal C) = \{ 1\}$, then $\C$
is pointed. Then $\C \simeq \C(G, \omega)$ for some abelian group
$G$ and some $3$-cocycle $\omega$ on $G$. Then $\C$ is solvable,
by \cite[Proposition 4.5 (ii)]{ENO2}.

\medbreak Suppose next that $\C$ is not pointed. Then all
non-invertible objects in $\C$ have Frobenius-Perron dimension
$p^m$, for some $m \geq 1$. Consider the group $G(\C)$ of
invertible objects of $\C$. Then $G(\C)$ is abelian and $G(\C)
\neq 1$, as follows by taking Frobenius-Perron dimensions in a
decomposition of the tensor product $X \otimes X^*$, for some
simple non-invertible object $X$.

Let us regard $\C$ as a premodular fusion category with respect to
its cano\-nical spherical structure (as $\FPdim \C$ is an integer).
Then $\C$ is modularizable, in view of \cite[Lemma
7.2]{tensor-exact}.

\medbreak Let $\widetilde \C$ be its modularization, which is a
modular category over $k$. Then $\C$ is an equivariantization $\C
\simeq \widetilde \C^G$ with respect to the action of a certain
group $G$ on $\widetilde \C$ \cite{bruguieres}. (Indeed, the
modularization functor $\C \to \widetilde  \C$ gives rise to an
exact  sequence of fusion categories $\Rep G \to \C \to \widetilde
\C$, which comes from an equivariantization; see  \cite[Example
5.33]{tensor-exact}.)

By construction of $G$, the category $\Rep G$ is the (tannakian)
fusion subcategory of transparent objects in $\C$. Therefore there
is an embedding of braided fusion categories $\Rep G \subseteq
\C$. In particular, the order of $G$ is odd, implying that $G$ is
solvable.

\medbreak By  Lemma \ref{cd-cg}, $\cd (\widetilde \C) \subseteq
\{p^m:\, m \geq 0\}$. Then, by induction, and since an
equi\-variantization of a solvable fusion category under the
action of a solvable group is again solvable, we may and shall
assume in what follows that $\C = \widetilde \C$ is modular.

\medbreak It is shown in \cite[Theorem 6.2]{gel-nik} that the
universal grading group $U(\C)$ is (abelian and) isomorphic to the
group $\widehat{G(\C)}$ of characters of $G(\C)$. In particular,
$U(\C) \neq 1$. On the other hand, $\C$ is a $U(\C)$-extension of
its fusion subcategory $\C_{\ad}$. Since also $\cd(\mathcal
C_{\ad}) \subseteq \{p^m:\, m \geq 0\}$, then $\C_{\ad}$ is
solvable, by induction. Therefore $\C$ is solvable, as
claimed.\end{proof}

\bibliographystyle{amsalpha}

\end{document}